\documentclass[11pt]{amsart}
\usepackage{amsmath}
\usepackage{tikz}

\theoremstyle{plain}
\newtheorem{theorem}{Theorem}[section]
\newtheorem{lemma}[theorem]{Lemma}

\newtheorem{corollary}[theorem]{Corollary}

\theoremstyle{definition}

\newtheorem{example}[theorem]{Example}
\numberwithin{equation}{section}

\def\be{\begin{equation}}
\def\ee{\end{equation}}

\begin{document}

\title[The Rigidity and Gap Theorem for Liouville's Equation]
{The Rigidity and Gap Theorem\\ for Liouville's Equation}
\author{Weiming Shen}
\address{Beijing International Center for Mathematical Research\\
Peking University\\
Beijing, 100871, China}  \email{wmshen@pku.edu.cn}

\begin{abstract}
In this paper, we study the properties of the first global term
in the polyhomogeneous expansions for Liouville's equation.
We obtain rigidity and gap results for the boundary integral of the global coefficient.
We prove that such a boundary integral is always nonpositive,
and is zero if and only if the underlying domain is a disc.
More generally, we prove some gap theorems relating such a boundary integral
to the number of components of the boundary.
The conformal structure plays an essential role.
\end{abstract}

\thanks{
The author acknowledges the support of NSFC
Grant 11571019.}
\maketitle

\section{Introduction}\label{sec-Intro}

Assume $\Omega\subset \mathbb{R}^{2}$ is a domain. We consider the following problem:
\begin{align}
\label{eq-MainEq} \Delta{u}& =e^{ 2u } \quad\text{in }\Omega, \\
\label{eq-MainBoundary}u&=\infty\quad\text{on }\partial \Omega.
\end{align}
The equation \eqref{eq-MainEq} is known as Liouville's equation.
Geometrically,  $e^{2u}(dx\otimes dx+dy\otimes dy)$ is
a complete metric with constant Gauss curvature $-1$ on $\Omega$.
For a large class of domains $\Omega$, \eqref{eq-MainEq} and \eqref{eq-MainBoundary} admit a
solution $u\in C^\infty(\Omega)$. The higher dimensional counterpart
is the singular Yamabe problem. For a given $(\Omega,g)$, an
$n$-dimensional smooth compact Riemannian manifold with boundary, with $n\geq3$,
we consider
\begin{align}
\label{eq-MEq u} \Delta_{g} u -\frac{n-2}{4(n-1)}S_gu&= \frac14n(n-2) u^{\frac{n+2}{n-2}} \quad\text{in }\,\Omega,\\
\label{eq-MBoundary u}u&=\infty\quad\text{on }\partial \Omega,
\end{align}
where $S_g$ is the scalar curvature of $\Omega$. Then, $u^{\frac{4}{n-2}}g$ is
the complete metric with a constant scalar curvature $-n(n-1)$ on $\Omega$.
According to Loewner and Nirenberg \cite{Loewner&Nirenberg1974} for domains in the Euclidean space and
Aviles and McOwen \cite{AM1988DUKE} for the general case,
\eqref{eq-MEq u}-\eqref{eq-MBoundary u} admits a unique positive solution.

There are many works concerning the boundary behaviors of
\eqref{eq-MainEq}-\eqref{eq-MainBoundary} and \eqref{eq-MEq u}-\eqref{eq-MBoundary u}.
Loewner and Nirenberg \cite{Loewner&Nirenberg1974} studied asymptotic behaviors
of solutions of \eqref{eq-MEq u} and \eqref{eq-MBoundary u} for domains in Euclidean space and proved
an estimate involving leading terms. Kichenassamy
\cite{Kichenassamy2004JFA, Kichenassamy2005JFA} expanded further under the assumption
that $\Omega$ has a $C^{2,\alpha}$-boundary
by establishing Schauder
estimates for degenerate elliptic equations of Fuchsian type.
When $\Omega$ has a smooth boundary,
Andersson, Chru\'sciel and Friedrich \cite{ACF1982CMP} and Mazzeo \cite{Mazzeo1991} established an
estimate up to an arbitrary finite order.
In fact, they proved that solutions of \eqref{eq-MEq u} and \eqref{eq-MBoundary u}
are polyhomogeneous. In \cite{Graham2017}, Graham studied the volume renormalization for singular
Yamabe metrics and characterized the coefficient of the first
logarithmic term in the polyhomogeneous expansion
for the conformal factor of singular Yamabe metric  by a variation of the coefficient of the
first logarithmic term in the volume expansion.
When the boundary is singular, Han and the author \cite{HanShen,HanShen2} studied the asymptotic
behaviors of solutions of \eqref{eq-MainEq}-\eqref{eq-MainBoundary} and \eqref{eq-MEq u}-\eqref{eq-MBoundary u}
for domains in the Euclidean space, and proved that these solutions are well approximated by the
corresponding solutions in tangent cones near singular points on the boundary.

Other geometric problems with a similar feature include complete K\"ahler-Einstein metrics discussed by
Cheng and Yau \cite{ChengYau1980CPAM},
Lee and Melrose \cite{LeeMelrose1982},
the complete minimal graphs in the hyperbolic space by Han and Jiang \cite{HanJiang} and
Lin \cite{Lin1989Invent}, 
complete conformal metrics of negative Ricci curvature
by Gursky, Streets, and Warren \cite{M.Gursky1}, and
hyperbolic affine spheres by Jian and Wang \cite{JianWang2013JDG}.

In this paper,
we study the properties of the global term in the polyhomogeneous
expansions for Liouville's equation. In certain cases, we can obtain geometric
and topological properties of the underlying domain from its integral.
In particular, we obtain some rigidity and gap results for the boundary integral of the global coefficient.

Let $\Omega$ be a bounded domain in $\mathbb R^2$ and
$u\in C^2(\Omega)$ be a solution of \eqref{eq-MainEq}-\eqref{eq-MainBoundary}.
Set
$$v=e^{-u}.$$
Then, $v$ satisfies
\begin{align}\label{eq-MEq-v}
v\Delta v &= |\nabla v|^2-1\quad\text{in }\Omega, \\
\label{eq-MBoundary-v}v&=0\quad\text{on }\partial \Omega.
\end{align}
Suppose $\Omega$ is a bounded $C^{3,\alpha}$ domain, then near $\partial \Omega$, $v$ has an expansion given by
\begin{equation}\label{boundary expansion v-1}v(x)=d(x)-\frac{1}{2}\kappa(y) d^2(x)+c_3(y)d^3(x)+
O(d^{3+\alpha}(x)),\end{equation}
where $d(x)$ is the distance from $x$ to $\partial\Omega$
and $\kappa(y)$ is the curvature of $\partial\Omega$ at $y\in\partial\Omega$ with
$|y-x|=d(x)$.
In particular,
$|\nabla v|=1$ on $\partial \Omega$. Hence, $v$ is a defining function on $\Omega$.
In this sense, $v$ is a function more natural to study than $u$.
We will formulate our main results for $v$ instead of $u$.

We note that  ${c}_3$ in \eqref{boundary expansion v-1} is
the coefficient of the first global term which has
no explicit expressions in terms of local geometry of $\partial \Omega$ and it is not conformally invariant.
Moreover, if $\partial\Omega$ is smooth, then, $v$ can be expanded to an arbitrary finite order term:
$$v=d-\frac{1}{2}\kappa d^2+{c}_3 d^3+{c}_4d^4+....$$
The coefficients of $d^i$ for $i\geq4$ can be expressed in terms of ${c}_3$ and the curvature of $\partial \Omega$.

Our first result is the following theorem.

\begin{theorem}\label{u disc}
Let $\Omega$ be a bounded $C^{3,\alpha}$ domain in $\mathbb R^2$, for some
$\alpha\in(0,1)$, and $\partial \Omega=\bigcup_{i=1}^{k}{\sigma}_{i}$,
where ${\sigma}_i$ is a simple closed $C^{3,\alpha}$ curve, $i=1,...,k$.
Let $v$ be the solution of \eqref{eq-MEq-v}-\eqref{eq-MBoundary-v} in $\Omega$
and $c_3$ be the coefficient of the first global term for $v$.
Then, for any $i$, $$\int_{\sigma_i}c_3dl\leq 0.$$ Moreover, if for some $i$,
$$\int_{{\sigma}_i}c_3dl=0,$$then,
$k=1$ and $\Omega$ is a bounded disc.
\end{theorem}

We now discuss whether we can obtain more information from $\int_{\sigma_i}c_3dl$. To this end,
we consider the normalized integrals $ \int_{\sigma_i}dl \int_{\sigma_i}c_3dl$ and
$ \int_{\partial \Omega}dl \int_{\partial \Omega}c_3dl$, which are invariant under rescaling.
We note that, for bounded convex domains, $-\int_{\sigma_i}dl \int_{\sigma_i}c_3dl$ can be arbitrarily large.
(See Example \ref{global term convex domain}.)
Next, we demonstrate that, in a multiply connected domain,
$ \int_{\sigma_i}dl \int_{\sigma_i}c_3dl$ indeed has a uniformly negative upper bound. We have the following gap theorem.

\begin{theorem}\label{not Simply connected domains}
Let $\Omega$ be a $k$-connected bounded $C^{3,\alpha}$ domain in $\mathbb R^2$, for $\alpha\in(0,1)$
and $k\geq2$, and  $\partial \Omega=\bigcup_{i=1}^{k}\sigma_{i}$, where $\sigma_i$
is a simple closed $C^{3,\alpha}$ curve, $i=1,...,k$.
Let $v$ be the solution of \eqref{eq-MEq-v}-\eqref{eq-MBoundary-v} in $\Omega$
and $c_3$ be the coefficient of the first global term for $v$.
Then, for each $i$,
\begin{equation}\label{gap for muti conn dom}\int_{\sigma_i}dl\int_{\sigma_i}c_3dl<-\frac{2\pi^2}{3}.\end{equation}
\end{theorem}

We point out that $ \int_{\sigma_i}dl \int_{\sigma_i}c_3dl$
is {\it not} conformally invariant and that the upper bound $-{2\pi^2}/{3}$ in the right-hand side
of \eqref{gap for muti conn dom} is optimal in multiply connected domains.
See Example \ref{optimal for k connected doamins}.

As consequences, we have the following
rigidity results.

\begin{theorem}\label{Simply connected domains rigidity}
Let $\Omega$ be a bounded $C^{3,\alpha}$ domain in $\mathbb R^2$, for some $\alpha\in(0,1)$,
and $\partial \Omega=\bigcup_{i=1}^{k}\sigma_{i}$, where $\sigma_i$ is a simple closed $C^{3,\alpha}$ curve, $i=1,...,k$.
Let $v$ be the solution of \eqref{eq-MEq-v}-\eqref{eq-MBoundary-v} in $\Omega$
and $c_3$ be the coefficient of the first global term for $v$.
Suppose that there exists some $i$ such that
$$\int_{\sigma_i}dl\int_{\sigma_i}c_3dl\geq-\frac{2\pi^2}{3}.$$
Then,
$k=1$ and $\Omega$ is a simply connected domain.
\end{theorem}

\begin{theorem}\label{k-connected domains rigidity}
Let $\Omega$ be a bounded $k$-connected $C^{3,\alpha}$ domain in $\mathbb R^2$,
for $\alpha\in(0,1)$
and $k\geq1$.
Let $v$ be the solution of \eqref{eq-MEq-v}-\eqref{eq-MBoundary-v} in $\Omega$
and $c_3$ be the coefficient of the first global term for $v$.
Suppose that there exists a positive integer $l\geq2$ such that
$$\frac{3}{2\pi^2}\int_{\partial \Omega}dl\int_{\partial \Omega}c_3dl\geq-l^2.$$
Then,
$k<l$.
\end{theorem}

One of the widely studied global quantities in conformal geometry is the renormalized volume
for the conformally compact Einstein metrics \cite{Graham1999}, and also for the singular Yamabe metrics \cite{Graham2017}.
The renormalized volume is conformally invariant for the conformally compact Einstein metrics,
but in general is not conformally invariant for the singular Yamabe metrics.
When $M$ is a $4$ dimensional conformally compact Einstein manifold,
the renormalized volume for $M$ can be expressed
as a linear combination of the integral of the $L^2$-norm of Weyl tensor and the  Euler number of $M$.
Relying on the conformal compactification, simple doubling and the work of Chang, Gursky and Yang
\cite{ChangGurskyYang2002, ChangGurskyYang2003} on a closed 4-manifold
with positive scalar curvature and large integral of $\sigma_2$ relating to the Euler number,
Chang, Qing and Yang \cite{ChangQingYang2004} obtained some topological information
from the renormalized volume of $4$ dimensional conformally compact Einstein manifold.
Recently, using blowup methods,
Li, Qing and Shi \cite{LiQingShi2017} obtained a gap theorem for a class
of $4$ dimensional conformally compact Einstein manifolds with very large renormalized volumes
under the condition that the first nonlocal term in the expansions of the metric near boundary is $0$.
This improves the work of Chang, Qing and Yang in \cite{ChangQingYang2007}.
With this gap theorem, Li, Qing and Shi \cite{LiQingShi2017}  obtained the rigidity of hyperbolic space.

In our case, the normalized integral of the coefficient of the first global term is
{\it not} conformally invariant, which causes major difficulties.
Our strategy is to connect directly the {\it boundary} integral
of the coefficient of the first global term with a {\it domain} integral of
a combination of lower order derivatives of $v$.
Then, based on a combination of conformal transforms and the maximum principle, we can
compare the integral of the coefficient of the first global term for the solution
in the underlying domain with the corresponding integral in certain model domains.

The paper is organized as follows. In Section \ref{sec-solutions and formulas},
we provide solutions
of \eqref{eq-MainEq}-\eqref{eq-MainBoundary} in some model domains
and derive an important formula for the integral of the coefficient of the first global term.
In Section \ref{sec-disc}, we study the sign of such an integral
on each boundary curve and prove Theorem \ref{u disc}.
In Section \ref{sec-Simply connected domains}, we study such an integral
in multiply connected domains and prove Theorem \ref{not Simply connected domains}.

The author would like to thank Qing Han
for introducing the problem studied in this paper and for his persistent encouragement.
The author would also like to thank Jie Qing and Gang Tian for helpful discussions.

\section{Important Solutions and Formulas}\label{sec-solutions and formulas}

In this section,
we will derive solutions of \eqref{eq-MainEq}-\eqref{eq-MainBoundary} in some model domains
and derive an important formula for the boundary integral of the first global term for $v$.

Throughout this paper, we will adopt notations from complex analysis and denote
 by $z = (x,y)$ points in the plane.

First,
we collect two well-known results concerning
solutions of \eqref{eq-MainEq}-\eqref{eq-MainBoundary}.

Let $\Omega_1$ and $ \Omega_2$ be two bounded domains in $\mathbb R^2$
and let $u_1$  and $u_2$
be the solutions of \eqref{eq-MainEq}-\eqref{eq-MainBoundary}
in $\Omega_{1}$ and $\Omega_{2}$, respectively. If $\Omega_1\subseteq\Omega_2$,
by the maximum principle, we have $u_1\geq u_2$ in $\Omega_1$.
Hence,
\begin{equation}\label{xaximum princ v}v_1=e^{-u_1}\leq e^{-u_2}=v_2\quad\text{in }\Omega_1.\end{equation}
Suppose $\sigma$ is a closed simple $C^{3,\alpha}$ curve,
$\sigma\subseteq\partial \Omega_1\bigcap \partial\Omega_2$.
Let $c_3^{1}$ and $c_3^2$ be the corresponding first global terms for $v_{1}$ and $v_2$ on $\sigma$,
respectively. By \eqref{xaximum princ v}, we have
\begin{equation}\label{xaximum princ c3}c_3^1\leq c_3^2\quad\text{on }\sigma.\end{equation}

Next, let $ \Omega_1$ and $\Omega_2$ be two bounded domains in $\mathbb R^2$.
Suppose $v_2\in C^{2}(\Omega_2)$ is a solution of \eqref{eq-MEq-v}-\eqref{eq-MBoundary-v}
in $\Omega_2$
and $w=x'+iy'=f(z)$ is a one-to-one holomorphic function
from $\Omega_1$ onto $\Omega_2$. Then,
\begin{equation}\label{slu under conf trans}v_1(z)=\frac{v_2(f(z))}{|f_{z}(z)|}\end{equation}
 is a solution of \eqref{eq-MEq-v} in $\Omega_1$.
In fact, we first note that $g_2=v_{2}^{-2}(dx\otimes dx+dy\otimes dy)$ is
a metric with constant Gauss curvature $-1$ on $\Omega_{2}$
and the Gauss curvature of the pull-back metric remains the same
under the conformal mapping.
Then, $g_1=f^{*}g_2=v_{1}^{-2}(dx\otimes dx+dy\otimes dy)$ is
a metric with constant Gauss curvature $-1$ on $\Omega_{1}$. Hence, $v_1$ solves \eqref{eq-MEq-v}.
\smallskip

We now give some important
solutions of \eqref{eq-MEq-v}-\eqref{eq-MBoundary-v}.

\begin{example}\label{exa-disc}
Consider $\Omega=B_R(0)$, for some $R>0$. Denote by $v_{R}$ the corresponding solution of
\eqref{eq-MEq-v}-\eqref{eq-MBoundary-v}. Then,
\begin{equation}\label{eq-solution-inside}v_{R}(z)=\frac{R^2-|z|^2}{2R}.\end{equation}
With $d(z)=R-|z|$, we have
$$v_{R}=d-\frac{d^2}{2R}.$$
\end{example}

\begin{example}\label{exa-disc-complement} Consider $\Omega=\mathbb R^2\setminus\overline{ B_R(0)}$ and let
\begin{equation}\label{eq-solution-outside}v_{R-}(z)=\frac{|z|^2-R^2}{2R}.\end{equation}
Then, $v_{R-}$ is a solution of
\eqref{eq-MEq-v}-\eqref{eq-MBoundary-v} for $\Omega$.
With $d(z)=|z|-R$, we have
$$v_{R-}=d+\frac{d^2}{2R}.$$
\end{example}

\begin{example}\label{exa-annular}
Consider $\Omega=B_{\frac{1}{R}}(0)\backslash \overline{B_{R}(0)}$, for some $R\in (0,1)$.
Let $v_{R,\frac{1}{R}}$ be the solution of \eqref{eq-MEq-v}-\eqref{eq-MBoundary-v} in $\Omega$.
Set $v_{R,\frac{1}{R}}=|z|h$ and $|z|=e^{-t}$. Then, $h=h(t)$ satisfies
\begin{align}
&hh_{tt}=h_{t}^{2}-1\quad\text{in }(-\ln\frac{1}{R},\ln\frac{1}{R}),\\
&h(\ln\frac{1}{R})=h(-\ln\frac{1}{R})=0.
\end{align}
In fact, $h$ is given by
$$h=\frac{2}{\pi}\ln\frac{1}{R}\cos\frac{\pi t}{2\ln\frac{1}{R}}.$$
Therefore, the solution of \eqref{eq-MEq-v}-\eqref{eq-MBoundary-v} for
$\Omega=B_{\frac{1}{R}}(0)\backslash \overline{B_{R}(0)}$ is given by
$$v_{R,\frac{1}{R}}=|z|\frac{2}{\pi}\ln\frac{1}{R}\sin\frac{\ln\frac{|z|}{R}}{\frac{2}{\pi}\ln \frac{1}{R}}.$$
Near $\partial B_R(0)$, $d(z)= |z|-R$ and
$$v_{R,\frac{1}{R}}=d+\frac{1}{2R}d^2+\bigg(-\frac{1}{6R^2}-\frac{1}{6(\frac{2}{\pi}R\ln\frac{1}{R})^2}\bigg)d^3+O(d^4).$$
Near $\partial B_{\frac{1}{R}}(0)$, $d(z)= \frac{1}{R}-|z|$ and
$$v_{R,\frac{1}{R}}=d-\frac{R}{2}d^2+\bigg(-\frac{R^2}{6}-\frac{R^2}{6(\frac{2}{\pi}\ln\frac{1}{R})^2}\bigg)d^3+O(d^4).$$
\end{example}

\begin{example}\label{exa-disc-minus-origin}
Consider $\Omega=B_{1}(0)\backslash\overline{ B_{R}(0)}$, for some $R\in (0,1)$.
By a rescaling in Example \ref{exa-annular}, the solution of \eqref{eq-MEq-v}-\eqref{eq-MBoundary-v}
for $\Omega=B_{1}(0)\backslash \overline{B_{R}(0)}$ is given by
$$v_{R,1}=|z|\frac{1}{\pi}\ln\frac{1}{R}\sin\frac{\ln\frac{1}{|z|}}{\frac{1}{\pi}\ln \frac{1}{R}}.$$
Let $R\rightarrow0$, $v_{R,1}\rightarrow|z|\ln\frac{1}{|z|}$,
which is a solution of \eqref{eq-MEq-v}-\eqref{eq-MBoundary-v} for $\Omega=B_1(0)\backslash \{0\}$.
\end{example}

\begin{example}\label{exa-disc-minus-point}
Consider $\Omega=B_1(0)\backslash \{z_0\}$ for some $z_0\in B_1(0)$.
Note that $\frac{-z+z_0}{1-\bar{z}_0z}$ is a holomorphic automorphism of $B_1(0)$.
Then, $$v_{1,z_0}=-|z-z_0|\ln |\frac{-z+z_0}{1-\bar{z}_0z}|.$$
is a solution of \eqref{eq-MEq-v}-\eqref{eq-MBoundary-v} for $\Omega=B_1(0)\backslash \{z_0\}$.
\end{example}

\begin{example}\label{exa-disc-complement2}
Consider $\Omega=B_{\frac{1}{R}}(0)\backslash \overline{B_{1}(0)}$, for some $R\in (0,1)$.
By a rescaling in Example \ref{exa-annular}, the solution of \eqref{eq-MEq-v}-\eqref{eq-MBoundary-v}
for $\Omega=B_{\frac{1}{R}}(0)\backslash\overline{ B_{1}(0)}$
is given by
$$v_{1,\frac{1}{R}}=|z|\frac{1}{\pi}\ln \frac{1}{R}\sin\frac{\ln |z|}{\frac{1}{\pi}\ln \frac{1}{R}}.$$
Letting $R\rightarrow0$, $v_{1,\frac{1}{R}}\rightarrow v_{1,\infty}:=|z|\ln |z|$,
which is a solution of \eqref{eq-MEq-v}-\eqref{eq-MBoundary-v} for $\Omega=\mathbb R^2\setminus\overline{ B_1(0)}$.
Recall that $v_{1-}(z) =\frac{1}{2}(|z|^2-1)$ is also a solution of \eqref{eq-MEq-v}-\eqref{eq-MBoundary-v}
for $\Omega=\mathbb R^2\setminus\overline{ B_1(0)}$, as in \eqref{eq-solution-outside} for $R=1$.
These are two different solutions  of \eqref{eq-MEq-v}-\eqref{eq-MBoundary-v} for $\Omega=\mathbb R^2\setminus\overline{ B_1(0)}$.
We note that $v_{1-}(\frac{1}{z})|z|^2$ is
 a solution of \eqref{eq-MEq-v}-\eqref{eq-MBoundary-v} in $B_1(0)$, but $v_{1,\infty}(\frac{1}{z})|z|^2$ is
a solution of \eqref{eq-MEq-v}-\eqref{eq-MBoundary-v} in $\Omega=B_1(0)\backslash \{0\}$.
\end{example}

Let $\Omega$ be a bounded $C^{3,\alpha}$ domain in $\mathbb R^2$, for some $\alpha\in(0,1)$.
Let $v$ be the solution of \eqref{eq-MEq-v}-\eqref{eq-MBoundary-v} in $\Omega$
and $c_3$ be the coefficient of the first global term for $v$.
Then, $v$ has the following expansion:
\begin{equation}\label{boundary expansion v}
v=d-\frac{1}{2}\kappa_{\partial\Omega}d^2+c_3d^3+O(d^{3+\alpha}).\end{equation}
We now derive an important formula for the boundary integral of $c_3$.

\begin{lemma}\label{lemma-integral-c3}
Let $\Omega\subseteq\mathbb R^2$ be a bounded $C^{3,\alpha}$ domain and
let $v$
be the solution of \eqref{eq-MEq-v}-\eqref{eq-MBoundary-v}
in $\Omega$. Let $c_3$ be the the coefficient of the first global term for $v$.
Then,
$$\int_{\partial \Omega}c_3 dl=-\int_{\Omega}\frac{(v_{xx}-v_{yy})^2+4v_{xy}^2}{6v}dx\wedge dy.$$
\end{lemma}

\begin{proof} We follow \cite{HanShen}. By \eqref{boundary expansion v}, we have
$$\Delta v=-2\kappa+6c_3d+o(d).$$
By applying  the Laplacian operator to \eqref{eq-MEq-v}, we get
$$\Delta (\Delta v)=\frac{1}{v}( 2|\nabla^2 v|^2-(\Delta v)^2)=\frac{(v_{xx}-v_{yy})^2+4v_{xy}^2}{v}.$$
Therefore,
\begin{align*}\int_{\partial \Omega}c_3 dl&=\frac{1}{6}\int_{\partial \Omega}\frac{\partial}{\partial d}(\Delta v)dl
=-\frac{1}{6}\int_{\partial \Omega}\frac{\partial}{\partial n}(\Delta v)dl
=-\int_{\Omega}\frac{1}{6} \Delta (\Delta v)dx\wedge dy\\
&=-\int_{\Omega}\frac{1}{6v} (2|\nabla^2 v|^2-(\Delta v)^2)dx\wedge dy\\
&=-\int_{\Omega}\frac{(v_{xx}-v_{yy})^2+4v_{xy}^2}{6v}dx\wedge dy,
\end{align*}
where $n$ is the unit outer normal vector of $\partial\Omega$.
\end{proof}

Let $ \Omega_1$ and $\Omega_2$ be two bounded $C^{3,\alpha}$ domains in $\mathbb R^2$,
for some $\alpha\in(0,1)$, and $v_1$ and $v_2$ be the solutions of \eqref{eq-MEq-v}-\eqref{eq-MBoundary-v}
in $\Omega_1$ and $\Omega_2$, respectively. Let $c_{3}$ be the coefficient of the first global term for $v_2$.
Suppose $w=x'+iy'=f(z)$ is a one-to-one holomorphic function
from $\Omega_1$ onto $\Omega_2$.  Then,
\begin{align}\label{v-1 v-2 transform}\begin{split}
v_2(w)&=v_1(z)|f_{z}(z)|,\\
\partial_wv_2&=\partial_zv_{1}|f_{z}|^{-1}\bar{f_{z}}+\frac{1}{2}v_{1}|f_{z}|^{-1}f_{z}^{-1}\bar{f_{z}}f_{zz},\\
\partial^2_wv_2&=\partial^2_zv_{1}|f_{z}|^{-1}f_{z}^{-1}\bar{f_{z}}+\frac{1}{2}v_{1}|f_{z}|^{-1}f_{z}^{-2}\bar{f_{z}}f_{zzz}
-\frac{3}{4}v_{1}|f_{z}|^{-1}f_{z}^{-3}\bar{f_{z}}f_{zz}^2\\
&=\partial^2_zv_{1}|f_{z}|^{-1}f_{z}^{-1}\bar{f_{z}}+\frac{1}{2}v_1|f_{z}|^{-1}f_{z}^{-1}\bar{f_{z}}\mathrm{S}(f),\end{split}
\end{align}
where
\begin{equation}\label{eq-Schwarzian}\mathrm{S}(f)=\frac{f_{zzz}}{f_{z}}-\frac{3}{2} (\frac{f_{zz}}{f_{z}})^2.\end{equation}
This is the Schwarzian derivative of $f$. By Lemma \ref{lemma-integral-c3}, we have
\begin{align*}
\int_{\partial \Omega_2}c_3 dl&=-\int_{\Omega_2}\frac{ 2|\nabla_{w}^2v_2|^2-(\Delta_{w} v_2)^2}{6v_2}dx'\wedge dy'\\
&=-\int_{\Omega_2}\frac{ (v_{2x'x'}-v_{2y'y'})^2+4v_{2x'y'}^2}{6v_2}dx'\wedge dy'\\
&=-\int_{\Omega_2}
\frac{ 4(v_{2ww}+v_{2\bar{w}\bar{w}})^2-4(v_{2ww}-v_{2\bar{w}\bar{w}})^2}{6v_2}\frac{i}{2}dw\wedge d\bar{w}.
\end{align*}
By \eqref{v-1 v-2 transform}, the above integral reduces
\begin{align*}
&=-\int_{\Omega_1}\frac{ 8\big[\text{Re}\big(v_1(z)|f_{z}(z)|\big)_{ww} \big]^2
+8\big[\text{Im}\big(v_1(z)|f_{z}(z)|\big)_{ww} \big]^2}{3v_1(z)|f_{z}(z)|}\frac{i|f_{z}(z)|^2}{2}dz\wedge d\bar{z}\\
&=-\int_{\Omega_1}\frac{ 8\big|\big(v_1(z)|f_{z}(z)|\big)_{ww} \big|^2}{3v_1(z)}|f_{z}(z)|dx\wedge dy.
\end{align*}
Hence,
\begin{align}\label{global term under transform}
\int_{\partial \Omega_2}c_3 dl
=-\int_{\Omega_1}\frac{ 8\big|\big(v_1(z)|f_{z}(z)|\big)_{ww} \big|^2}{3v_1(z)}|f_{z}(z)|dx\wedge dy.
\end{align}

\begin{corollary}
Let $\Omega\subseteq\mathbb R^2$ be a simply connected bounded $C^{3,\alpha}$ domain and
let $v$
be the solution of \eqref{eq-MEq-v}-\eqref{eq-MBoundary-v}
in $\Omega$. Let $c_3$ be the coefficient of the first global term for $v$.
Suppose $w=f(z)$ is a one-to-one holomorphic function
from $B_1(0)$ to $\Omega$. Then,
$$\int_{\partial \Omega}c_3 dl=-\int_{B_1(0)}\frac{(1-|z|^{2})|\mathrm{S}(f) |^2}{3|f_{z}|}dx\wedge dy.$$
where $\mathrm{S}(f)$ is the Schwarzian derivative of $f$ given by \eqref{eq-Schwarzian}.
\end{corollary}

\begin{proof}
Note that $\frac{1}{2}(1-|z|^{2})$ is the solution of \eqref{eq-MEq-v}-\eqref{eq-MBoundary-v}
in $B_1(0)$ and $(\frac{1-|z|^{2}}{2})_{zz}=0$. Then, by \eqref{global term under transform},
$$\int_{\partial \Omega}c_3 dl=-\int_{B_1(0)}(\frac{1-|z|^{2}}{3})|f_{z}|^{-1}|\mathrm{S}(f) |^2dx\wedge dy.$$
This is the desired result.\end{proof}

\section{The Rigidity}\label{sec-disc}

In this section, we study the sign of the integral of the coefficient of the first global term for $v$
on each boundary curve. We will prove
such an integral on each boundary curve is negative unless the underlying domain is a disc in $\mathbb R^2$.

We first prove Theorem \ref{u disc}.

\begin{proof}[Proof of Theorem \ref{u disc}]
Without loss of generality, we assume $\sigma_1$ is the boundary of the
unbounded component of $\mathbb R^2\backslash \overline{\Omega}$.
By Lemma \ref{lemma-integral-c3}, we have
$$\sum_{i=1}^{k}\int_{\sigma_i}c_3 dl=-\int_{\Omega}\frac{(v_{xx}-v_{yy})^2+4v_{xy}^2}{6v}dx\wedge dy\leq 0.$$

If $\sum_{i=1}^{k}\int_{\sigma_i}c_3 dl=0$,
we must have $$ v_{xx}=v_{yy},\, v_{xy}= 0\quad\text{in }\Omega.$$
Then, $v=a+b_1x+b_2y+c(x^2+y^2)$ for
some constants $a$, $b$, $c$. Without loss of generality, we assume $v$ takes its maximum at $0\in \Omega$. Then, $v=a+c(x^2+y^2).$ Since $\Omega$ is bounded and $v=0$ on $\sigma_1\subseteq \partial \Omega$, we have
$$a>0,\quad c<0,$$ and
$$\sigma_1=\partial B_{\sqrt{-\frac{a}{c}}}(0).$$
In other words, $\sigma_1$ encloses a disc in $\mathbb{R}^{2}$.


Let $\Omega_1$ be the bounded domain enclosed by $\sigma_1$ and $\Omega\subseteq\Omega_1$.
Let $v_{1}$ be the corresponding solution of \eqref{eq-MEq-v}-\eqref{eq-MBoundary-v} in $\Omega_1$ and
$c_{3}^{1}$ be the coefficient of the first global term for $v_{1}$.
A similar argument yields
 $\int_{\sigma_1}c_3^{1}dl\leq 0$, and
$\int_{\sigma_1}c_3^{1}dl=0$ if and only if $\Omega_1$ is a bounded disc.
By the maximum principle, we have
$v\leq v_{1}$. Therefore, we have $c_3\leq c_{3}^{1}$ on $\sigma_1$ and hence,
$$\int_{\sigma_1}c_3dl\leq 0.$$
If $\int_{\sigma_1}c_3^{1}dl<0$, then $\int_{\sigma_1}c_3dl<0$.
If  $\int_{\sigma_1}c_3^{1} dl= 0,$ then $\Omega_1$ is a disc.
Without loss of generality, we can assume $\Omega_1=B_1(0)$. If $\Omega \neq \Omega_1$,
we can find some point $z_0 \in B_1(0)$ with $z_0\in \Omega^{c}$.
Set $$v_{z_0}=-|z-z_0|\ln |\frac{-z+z_0}{1-\bar{z}_0z}|.$$ Then,
$v_{z_0}$ is a solution of \eqref{eq-MEq-v}-\eqref{eq-MBoundary-v} in $B_1(0)\backslash\{z_0\}$.
Let $c_3^{z_0}$ be the coefficient of the  first global term for $v_{z_0}$.
Then, by the maximum principle, we have
$$\int_{\sigma_1}c_3 dl \leq\int_{\sigma_1}c_3^{z_0}dl< 0.$$
 Therefore, in both cases, we have $\int_{\sigma_1}c_3 dl \leq0$,
 and the equality holds if and only if $\Omega$ is a disc.

Now we consider $i\geq2$. Let $\Omega_i$ be the bounded domain enclosed by $\sigma_i$
and set $\widehat\Omega_{i}=\mathbb R^2\setminus\overline{\Omega}_i$.
Fix a point $z_0\in\Omega_i$.
Consider the curve
$$\widetilde{\sigma}_i= \{z|z_0+\frac{z-z_0}{|z-z_0|^2} \in\sigma_i\}.$$
Let $\widetilde{\Omega}_i$ be the bounded domain enclosed by $\widetilde{\sigma}_i$ and
$\widetilde{v}_i$ be the solution of \eqref{eq-MEq-v}-\eqref{eq-MBoundary-v} in $\widetilde{\Omega}_i$.
Set
$$\widehat v_{i}=\big[\widetilde{v}_i(z_0+\frac{z-z_0}{|z-z_0|^2}) \big]|z-z_0|^2.$$
Then, $\widehat v^{i}$ is a positive solution of \eqref{eq-MEq-v} for
$\widehat\Omega_{i}$ and $\widehat v_{i}=0$ on $\sigma_i$.
Let $\widehat c_{3}^{\,i}$ be the coefficient of the first global term
for $\widehat v_{i}$. Then,
\begin{align*}&\int_{\sigma_i}\widehat c_3^{\,i} dl-\frac{1}{6}\int_{|x|=R}\frac{\partial}{\partial r}(\Delta \widehat v_{i})dl\\
&\qquad=-\int_{B_R(0)\setminus \Omega_{i}}
\frac{(\widehat v_{ixx}-\widehat v_{iyy})^2+4(\widehat v_{ixy})^2}{6\widehat v_{i}}dx \wedge dy \leq 0.\end{align*}
By a direct computation, we have
$$\lim_{R\rightarrow\infty}\int_{|x|=R}\frac{\partial}{\partial r}(\Delta \widehat v_{i})dl=0.$$
Letting $R\rightarrow\infty$, we get
$$\int_{\sigma_i}\widehat c_3^{\,i} dl=-\int_{\widehat\Omega_{i}}
\frac{(\widehat v_{ixx}-\widehat v_{iyy})^2+4(\widehat v_{ixy})^2}{6\widehat v_{i}} dx\wedge dy  \leq 0.$$
If $\int_{\sigma_i}\widehat c_3^{\,i} dl<0$, by the maximum principle, we have
$$\int_{\sigma_i}c_3 dl\leq \int_{\sigma_i}\widehat c_3^{\,i} dl<0.$$
If $\int_{\sigma_i}\widehat c_3^{\,i} dl=0$, then
$$\widehat v_{i}=a+2b_1x+2b_2y+2c(x^2+y^2).$$
Since $\widehat v_{i}>0$ for $|z|$ large and $\widehat v_{i}=0$ on  $\sigma_i$, we have
$$\frac{b_1^2}{c}+\frac{b_2^2}{c}-a>0,\quad c>0,$$ and
$$\sigma_{i}=\partial B_{\sqrt{-\frac{b_1^2+b_2^2-ac}{c^2}}}\big(-\frac{b_1}{c},-\frac{b_2}{c}\big).$$
Without loss of generality, we  assume $\sigma_{i}=\partial B_1(0)$.
Take $R$ sufficiently large such that $\Omega\subseteq B_R(0)\backslash\overline{ B_1(0)}$.
Let $v_{R}$ be the solution of \eqref{eq-MEq-v}-\eqref{eq-MBoundary-v} in $B_R(0)\backslash\overline{ B_1(0)}$ and
$c_3^{R}$ be the corresponding first global term for $v_{R}$.
By the maximum principle, we have
$$\int_{\sigma_i}c_3 dl \leq\int_{\sigma_i}c_3^{R}dl< 0.$$
In summary, for $i\geq2$, we conclude $\int_{\sigma_i}c_3 dl <0$.
\end{proof}

\section{Gap Theorems}\label{sec-Simply connected domains}

In this section, we study the integral of the coefficient of the first global term for $v$
in multiply connected domains. We will prove the normalized integral of the coefficient
of the first global term for $v$ in multiply connected domains has a negative upper bound
on each boundary curve.

Let $\Omega$ be a bounded smooth domain in $\mathbb R^2$.
Set $k\Omega=\{x|\frac{x}{k}\in \Omega\}$.
Let $v$ be the solution of \eqref{eq-MEq-v}-\eqref{eq-MBoundary-v} for $\Omega$
and let $c_3$ be the coefficient of the first global term for $v$.
Then, $v_{k}=kv(\frac{x}{k})$ is the corresponding
solution of \eqref{eq-MEq-v}-\eqref{eq-MBoundary-v} for $\Omega=k\Omega$.
Set $k\sigma_i=\{x|\frac{x}{k}\in \sigma_i\}$.
For $v_{k}$, near $\partial (k\Omega)$, we have the following expansion:
\begin{equation}\label{boundary expansion vk}v_k=d-\frac{1}{2}\kappa_{\partial (k\Omega)}d^2+c^{k\Omega}_3d^3+o(d^3).\end{equation}
Then, \begin{equation}\label{global term transform}c^{k\Omega}_3=\frac{c_3}{k^2}.\end{equation}
Hence, $ \int_{\sigma_i}dl \int_{\sigma_i}c_3dl$ is invariant under rescaling.
By Theorem \ref{u disc}, we always have
$$\int_{\sigma_i}dl \int_{\sigma_i}c_3dl\leq0.$$
The following example shows that in certain bounded convex domains,
$ -\int_{\sigma_i}dl \int_{\sigma_i}c_3dl$ may be arbitrarily large.

\begin{example}\label{global term convex domain}
For a fixed positive number $L$,
let $\Omega=\{(x,y)|-L<y<L\}$.
Then, $v=\frac{2L}{\pi}\cos\frac{\pi y}{2L}$ is a solution of \eqref{eq-MEq-v}-\eqref{eq-MBoundary-v}.
For $R>0$ sufficiently large, let $\Omega_R\subseteq \Omega\bigcap B_{2R}(0)$
be a bounded convex smooth domain coinciding  with $\Omega$ in $B_{R}$.
Let $v_{R}$ be the solution of \eqref{eq-MEq-v}-\eqref{eq-MBoundary-v} in $\Omega_R$.
Then, $v_R$ has the expansion:
\begin{equation}\label{boundary expansion vR}
v_R=d-\frac{1}{2}\kappa_{\partial \Omega_R}d^2+c^{R}_3d^3+o(d^3).\end{equation}
We will prove, as $R\rightarrow\infty$,
$$\int_{\partial \Omega_R}dl \int_{\partial \Omega_R}c_3^{R}dl\rightarrow -\infty.$$
In fact, a direct computation implies
 $$\Delta \ln\big(\frac{1}{\frac{2L}{\pi}\cos\frac{\pi y}{2L}}+\frac{2R}{R^2-|x|^2}\big)
 <\big(\frac{1}{\frac{2L}{\pi}\cos\frac{\pi y}{2L}}+\frac{2R}{R^2-|x|^2}\big)^2
 \quad\text{in }\Omega\cap B_R(0).$$
Then, by the maximum principle for the equation \eqref{eq-MainEq}, we have
 $$ -\ln\big(\frac{2L}{\pi}\cos\frac{\pi y}{2L}\big)\leq-\ln v_R\leq
 \ln\big(\frac{1}{\frac{2L}{\pi}\cos\frac{\pi y}{2L}}+\frac{2R}{R^2-|x|^2}\big)
 \quad\text{in }\Omega\cap B_R(0).$$
Hence,
 $$\frac{1}{ \frac{1}{\frac{2L}{\pi}\cos\frac{\pi y}{2L}}+\frac{2R}{R^2-|x|^2}}
 \leq v_R\leq \frac{2L}{\pi}\cos\frac{\pi y}{2L}
\quad\text{in }\Omega\cap B_R(0).$$ Then, it is easy to prove, for any $m$,
\begin{equation}v_R\rightarrow\frac{2L}{\pi}\cos\frac{\pi y}{2L}
\quad\text{in }C^{m}([-1,1]\times[-\frac{L}{2},\frac{L}{2}] )
\text{ as }R\rightarrow\infty.\end{equation}
Therefore,
\begin{align*}&\int_{\partial \Omega_R}dl \int_{\partial \Omega_R}c_3^{R}dl\\
&\qquad
<\int_{\partial \Omega_R}dl \int_{[-1,1]\times[-\frac{L}{2},\frac{L}{2}]}
\frac{(\Delta (\frac{2L}{\pi}\cos\frac{\pi y}{2L}))^2-2|\nabla^2
(\frac{2L}{\pi}\cos\frac{\pi y}{2L})|^2}{12(\frac{2L}{\pi}\cos\frac{\pi y}{2L})}\rightarrow-\infty,\end{align*}
as $R\rightarrow\infty$.
\end{example}

The following result shows that, if $\Omega$ is a $2$-connected domain,
then the normalized integral of the coefficient of the first global term for $v$
has a negative upper bound on each boundary curve.

\begin{theorem}\label{Simply connected domains}
Let $\Omega$ be a bounded $2$-connected $C^{3,\alpha}$ domain in $\mathbb R^2$, for some
$\alpha\in(0,1)$,
and $\partial \Omega=\sigma_1\cup\sigma_2$, where $\sigma_1$ and $\sigma_2$
are simple closed $C^{3,\alpha}$ curves.
Let $v$ be the solution of \eqref{eq-MEq-v}-\eqref{eq-MBoundary-v} in $\Omega$
and let $c_3$ be the coefficient of the first global term for $v$.
Then, for $i=1,2$,   $$\int_{\sigma_i}dl\int_{\sigma_i}c_3dl<-\frac{2\pi^2}{3}.$$
\end{theorem}

\begin{proof}
Without loss of generality, we assume $\sigma_1$ is the boundary
of the unbounded component of $\mathbb R^2\backslash\overline{ \Omega}$.
Denote by  $\Omega_i$ the bounded domain enclosed by $\sigma_i$, $i=1,2$.
Then, $\Omega_1$ and $\Omega_2$ are simply connected domains and $\Omega_2\subseteq\Omega_1$.

We first prove
$$\int_{\sigma_1}dl\int_{\sigma_1}c_3dl<-\frac{2\pi^2}{3}.$$
 Up to a Rescaling, we can assume $\int_{\sigma_1}dl=2\pi$ and we will prove
$$\int_{\sigma_1}c_3dl<-\frac{\pi}{3}.$$
 Without loss of generality, we assume $0\in \Omega_2$.

Note that $|z|\ln\frac{1}{|z|}$ is a solution of \eqref{eq-MEq-v}-\eqref{eq-MBoundary-v}
for $B_1(0)\setminus \{0 \}$.
Let $w=x'+iy'=f(z)$ be a one to one holomorphic map from $\overline{B_1(0)}$ on to $\overline{\Omega_1}$ with $f(0)=0$.
Consider
$$v^{+}(f(z))=\big(|z|\ln\frac{1}{|z|}\big)|f_{z}(z)|,$$
or
$$v^{+}(w)=\big(|f^{-1}(w)|\ln\frac{1}{|f^{-1}(w)|}\big)|f_{z}(f^{-1}(w))|.$$
Then, $v^{+}$ is a solution of \eqref{eq-MEq-v}-\eqref{eq-MBoundary-v} in $\Omega_1\backslash\{0\}$.
Let $c_3^{+}$ be the coefficient of the first global term for $v^+$ on $\sigma_1$.

By a direct computation, we have
\begin{align*}
\frac{\partial}{\partial r }\bigg(r\ln\frac{1}{r}\bigg)=-\ln r-1,\quad
\frac{\partial^2}{\partial r^2 }\bigg(r\ln\frac{1}{r}\bigg)=-\frac{1}{r},\end{align*}
and
\begin{align*} \Delta\bigg(r\ln\frac{1}{r}\bigg)=\frac{-\ln r-2}{r},\quad
\frac{\partial}{\partial r }\Delta\bigg(r\ln\frac{1}{r}\bigg)=\frac{\ln r+1}{r^2}.
\end{align*}
Then, at $r=\frac{1}{e}$, we have
$$\frac{\partial}{\partial r }\bigg(r\ln\frac{1}{r}\bigg)=0,\quad
\frac{\partial}{\partial r }\Delta\bigg(r\ln\frac{1}{r}\bigg)=0.$$
Let $n$ be the unit outer normal vector field on $\partial\{ f(|z|<\frac{1}{e})\}$.
Then,
\begin{align*}
\int_{f(|z|=\frac{1}{e})}\frac{\partial}{\partial n }\Delta_{w} v^{+} |dw|
 =\int_{|z|=\frac{1}{e}}
\frac{1}{|f_{z}|}\frac{\partial}{\partial r}\bigg[\frac{1}{|f_{z}|^2} \Delta_{z}  \bigg(|z|\ln\frac{1}{|z|}|f_{z}|\bigg)\bigg] |f_{z}| |dz|.
\end{align*}

For the integrand, we have on $|z|=1/e$,
\begin{align*}
&\,\frac{\partial}{\partial r}\bigg[\frac{1}{|f_{z}|^2} \Delta_{z}  \bigg(|z|\ln\frac{1}{|z|}|f_{z}|\bigg)\bigg] \\
=&
\,\frac{\partial}{\partial r}\bigg[\frac{1}{|f_{z}|}
\Delta_{z}  \bigg(|z|\ln\frac{1}{|z|}\bigg)-2\frac{\partial}{\partial r} \bigg(|z|\ln\frac{1}{|z|}\bigg)
\frac{\partial}{\partial r}\frac{1}{|f_{z}|}+ \bigg(|z|\ln\frac{1}{|z|}\bigg)\frac{|f_{zz}|^2}{|f_z|^3}
    \bigg] \\
=\,&e\frac{\partial}{\partial r}\frac{1}{|f_{z}|}+
\frac{1}{e}\frac{\partial}{\partial r} \frac{|f_{zz}|^2}{|f_z|^3}.\end{align*}
Hence,
\begin{align*}
\int_{f(|z|=\frac{1}{e})}\frac{\partial}{\partial n }\Delta_{w} v^{+} |dw| =\int_{|z|\leq\frac{1}{e}}\bigg(e\frac{|f_{zz}|^2}{|f_z|^3}+
\frac{4}{e} \frac{|\emph{S}(f)|^2}{|f_z|}
    \bigg)\frac{i}{2}dz\wedge d\bar{z}
     \geq 0,
\end{align*}
where $\emph{S}(f)$ is the Schwarzian derivative of the holomorphic
function defined by \eqref{eq-Schwarzian}.

Then, by a similar computation as in proving \eqref{global term under transform}, we have
\begin{align}
\begin{split}
&\int_{\sigma_1}c_3^{+}dl\\=&-\int_{\Omega_{1}
\backslash\overline{f( B_{\frac{1}{e}}(0))}}\frac{ 2|\nabla_{w}^2 v^{+}|^2
-(\Delta_{w} v^{+})^2}{6v^{+}}dx'\wedge dy'-\int_{f(|z|=\frac{1}{e})}\frac{\partial}{\partial n }\Delta v^{+} |dw|\\
\leq &-\int_{\Omega_{1} \backslash\overline{f( B_{\frac{1}{e}}(0))}}\frac{ 2|\nabla_{w}^2 v^{+}|^2-(\Delta_{w} v^{+})^2}{6v^{+}}dx'\wedge dy'\\
=&-\int_{B_1(0)\backslash\overline{ B_{\frac{1}{e}}(0)}}\frac{8| v^{+}_{ww}|^2}{3v^{+}}\frac{i|f_{z}|^2}{2}dz\wedge d\bar{z}\\
=& -\int_{B_1(0)\backslash\overline{ B_{\frac{1}{e}}(0)}}\frac{8\bigg|\big(|z|\ln\frac{1}{|z|}\big)_{zz}|f_{z}|^{-1}f_{z}^{-1}\bar{f_{z}}
+\frac{1}{2}|z|\ln\frac{1}{|z|}|f_{z}|^{-1}f_{z}^{-1}\bar{f_{z}}\emph{S}(f)\bigg|^2}{3|z|\ln\frac{1}{|z|}}|f_{z}|dx\wedge dy \\ =& -\int_{B_1(0)\backslash\overline{ B_{\frac{1}{e}}(0)}}\frac{8\bigg|\frac{\ln |z|}{4|z|}\frac{\bar{z}^2}{|z|^2}|f_{z}|^{-1}f_{z}^{-1}\bar{f_{z}}
+\frac{1}{2}|z|\ln\frac{1}{|z|}|f_{z}|^{-1}f_{z}^{-1}\bar{f_{z}}\emph{S}(f)\bigg|^2}{3|z|\ln\frac{1}{|z|}}|f_{z}|dx\wedge dy.\end{split}
 \end{align}
For the numerator, by writing $|c|^2=c\bar c$ and a straightforward computation, the above integral reduces to
$$
-\int_{B_1(0)\backslash\overline{ B_{\frac{1}{e}}(0)}}
\bigg[\frac{1}{6}\frac{|\ln |z||}{|z|^3|f_z|}+\frac{2|z|}{3}\ln\frac{1}{|z|}\frac{|\emph{S}(f)|^2}{|f_z|}-\frac{\ln \frac{1}{|z|}}{3|z|^3}
\frac{z^2\emph{S}(f)+ \overline{z^2\emph{S}(f)} }{|f_z|}\bigg]dx\wedge dy.$$
By \eqref{eq-Schwarzian}, we have
$$-\frac{1}{2}\frac{\emph{S}(f) }{|f_z|}=\frac{\partial^2}{\partial z^2}\frac{1}{|f_z|}.$$
Then, in polar coordinates,
 \begin{align*}&-\frac{z^2\emph{S}(f)+ \overline{z^2\emph{S}(f)} }{2|f_z|}
 = (z^2\frac{\partial^2}{\partial z^2}+\bar{z}^2\frac{\partial^2}{\partial \bar{z}^2})\frac{1}{|f_z|}\\
  =&\bigg[ (z\frac{\partial}{\partial z}-\bar{z}\frac{\partial^2}{\partial \bar{z}})^2
  +2z\bar{z}\frac{\partial}{\partial z}\frac{\partial}{\partial \bar{z} }  -(z\frac{\partial}{\partial z}+\bar{z}\frac{\partial}{\partial \bar{z}})\bigg]\frac{1}{|f_z|}\\
   =&\bigg(-\frac{\partial^2}{\partial \theta^2}+\frac{r^2}{2}\Delta -r\frac{\partial}{\partial r} \bigg )\frac{1}{|f_z|}.
 \end{align*}
Hence, for any $r\in (0,1]$, we have
 \begin{align*}
 \int_{|z|=r}-\frac{z^2\emph{S}(f)+ \overline{z^2\emph{S}(f)} }{2|f_z|}|dz|
 =  \frac{r^2}{2}\int_{|z|=r}\Delta\frac{1}{|f_z|}|dz| -r\int_{|z|\leq r}\Delta\frac{1}{|f_z|}dx\wedge dy.
 \end{align*}
Note
$$\Delta\bigg(\Delta \frac{1}{|f_z| }\bigg)=4\frac{|\emph{S}(f) |^2}{|f_z|}\geq 0.$$
Then, for any $ r'\in [0,r]$,
\begin{equation}\int_{0}^{2\pi}\Delta \frac{1}{|f_z| }(r',\theta)d\theta
=\frac{1}{r'}\int_{|z|=r'}\Delta \frac{1}{|f_z| }|dz|\leq \frac{1}{r}\int_{|z|=r}\Delta \frac{1}{|f_z| }|dz|.\end{equation}
Hence,
 \begin{align*}
 \int_{|z|=r}-\frac{z^2\emph{S}(f)+ \overline{z^2\emph{S}(f)} }{2|f_z|}|dz|
 \geq   \frac{r^2}{2}\int_{|z|=r}\Delta\frac{1}{|f_z|}|dz| -  \frac{r^2}{2}\int_{|z|=r}\Delta\frac{1}{|f_z|}|dz|
 =0
 \end{align*}
Therefore, we have
\begin{align*}\int_{\sigma_1}c_3^{+}dl \leq& -\int_{B_1(0)\backslash\overline{ B_{\frac{1}{e}}(0)}}
\bigg(\frac{1}{6}\frac{|\ln |z||}{|z|^3|f_z|}+\frac{2|z|}{3}\ln\frac{1}{|z|}\frac{|\emph{S}(f)|^2}{|f_z|} \bigg) dx\wedge dy\\
  \leq& -\int_{B_1(0)\backslash\overline{ B_{\frac{1}{e}}(0)}}\frac{1}{6}\frac{|\ln |z||}{|z|^3|f_z|} dx\wedge dy
\end{align*}
Note
$$|f_z(0)|=\bigg|\frac{1}{2\pi}\int_{|z|=1}f_{z}|dz|\bigg |\leq\frac{1}{2\pi}\int_{|z|=1}|f_{z}||dz|=1,$$
where the equality holds if and only if $f_z\equiv e^{i\varphi}$ for some constant $\varphi$; namely, $f$ is a rotation.
 Since $\Delta \frac {1}{|f_z|} \geq 0$, we have, for any $ r\in [0,1]$,
\begin{equation}\frac{1}{2\pi r}\int_{|z|=r}\frac {1}{|f_z|}|dz|\geq\frac {1}{|f_z(0)|} .\end{equation}
Then,
\begin{align*}\int_{\sigma_1}c_3^{+}dl
  \leq& -\int_{B_1(0)\backslash\overline{ B_{\frac{1}{e}}(0)}}\frac{1}{6}\frac{|\ln |z||}{|z|^3|f_z|} dx\wedge dy
  = -\int_{\frac{1}{e}}^{1}\bigg(\int_{0}^{2\pi}\frac{1}{6}\frac{|\ln r|}{r^3|f_z|}d\theta \bigg)rdr \\
 \leq & -\int_{\frac{1}{e}}^{1}\frac{\pi}{3}\frac{|\ln r|}{r^3|f_z(0)|}rdr
  =-\frac{\pi}{3|f_z(0)|}
 \leq -\frac{\pi}{3}.
\end{align*}
Therefore, we have
\begin{equation}\int_{\sigma_1}c_3^{+}dl
\leq  -\frac{\pi}{3}.\end{equation}
If $\int_{\sigma_1}c_3^{+}dl<-{\pi}/{3}$, then
$$\int_{\sigma_1}c_3dl\leq \int_{\sigma_1}c_3^{+}dl<-\frac{\pi}{3}.$$
If $\int_{\sigma_1}c_3^{+}dl=-{\pi}/{3}$, then $\sigma_1=\partial B_1(0)$ and $\Omega_{1}=B_1(0)$.
Hence, for some $\epsilon>0$ sufficiently small,
we have $\Omega\subseteq B_1(0)\backslash \overline{B_{\epsilon}(0)}.$
Let $v^{\epsilon}$
be the solution of \eqref{eq-MEq-v}-\eqref{eq-MBoundary-v}
in $B_1(0)\backslash\overline{ B_{\epsilon}(0)}$ and $c_3^{\epsilon}$ be the coefficient of the first global term for $v^{\epsilon}$.
Then, we have $$\int_{\sigma_1}c_3dl\leq \int_{\sigma_1}c_3^{\epsilon}dl<-\frac{\pi}{3}.$$
Hence, in both cases, we have
$$\int_{\sigma_1}c_3 dl<-\frac{\pi}{3}.$$
Therefore, $$\int_{\sigma_1} dl\int_{\sigma_1}c_3 dl<-\frac{2\pi^2}{3}.$$

\smallskip

We now prove
$$\int_{\sigma_2} dl\int_{\sigma_2}c_3 dl<-\frac{2\pi^2}{3}.$$ Up to a rescaling, we assume $\int_{\sigma_2}dl=2\pi$ and
proceed to prove
$$\int_{\sigma_2}c_3 dl<-\frac{\pi}{3}.$$
Set $\Omega_{2*}=\mathbb R^2\setminus\overline{\Omega}_2$.
Fix a point $z_0\in\Omega_2$, say $z_0=0$. Set
$$\widehat{\Omega}_2=\{z|\frac{1}{\bar{z}} \in\Omega_{2*} \}.$$
Let $\widetilde{f}(z)$ be a one-to-one holomorphic function
from $\overline{B_1(0)}$ onto $\overline{\widehat{\Omega_2}}$ with $\widetilde{f}(0)=0$. Then,
$$w=x'+iy'=f(z)=\frac{1}{\overline{\widetilde{f}(\frac{1}{\bar{z}})}}$$
is a one-to-one holomorphic function from $(B_1(0))^c$ onto $\overline{\Omega_{2*}}$.

Note that $v^{\infty}=|z|\ln |z|$ is a solution of \eqref{eq-MEq-v}-\eqref{eq-MBoundary-v} in $(B_1(0))^c$.
Set
$$v^-(f(z))=\big(|z|\ln |z|\big)|f_{z}(z)|,$$
or
$$v^-(w)=\big(|f^{-1}(w)|\ln|f^{-1}(w)|\big)|f_{z}(f^{-1}(w))|.$$
Then, $v^-$ is a solution of \eqref{eq-MEq-v}-\eqref{eq-MBoundary-v} in $\Omega_{z*}$.
Let $c_3^{-}$ be the coefficient of the first global term for $v^{-}$ on $\sigma_2$.

Assume
$$\widetilde{f}=a_1z+a_2z^2+a_3z^3+ O(|z|^4)\quad\text{near }0.$$
Then,
$$f=\frac{1}{\frac{a_1}{z}+\frac{a_2}{z^2}+\frac{a_3}{z^3}+O(\frac{1}{|z|^4})}=\frac{z}{a_1}-\frac{a_2}{a_1^2}+\big(\frac{a_2^2}{a_1^3}-\frac{a_3}{a_1^2}\big)\frac{1}{z}+O(\frac{1}{|z|^2})\quad\text{near } \infty.$$
Let $n$ be the unit outer normal vector field on $\partial\{ f(1<|z|<R)\}$. By a direct computation, we have
\begin{align*}&\lim_{R\rightarrow\infty}\int_{f(|z|=R)}\frac{\partial}{\partial n }\Delta_{w} v^{-} |dw|\\= &\lim_{R\rightarrow\infty}\int_{|z|=R}\frac{1}{|f_{z}|}\frac{\partial}{\partial r}\bigg[\frac{1}{|f_{z}|^2} \Delta_{z}  \bigg(|z|\ln\frac{1}{|z|}|f_{z}|\bigg)\bigg] |f_{z}| |dz|\\
=&\lim_{R\rightarrow\infty}O(\frac{\ln R}{R})=0. \end{align*}
Then, we have
 \begin{align}\begin{split}&\int_{\sigma_2}c_3^{-}dl \\
 =& \lim_{R\rightarrow\infty}\bigg[\int_{B_R(0)\backslash B_1(0)}
 \frac{8\big|\big((|z|\ln|z|)|f_{z}(z)|\big)_{ww} \big|^2|f_{z}|}{-3|z|\ln|z|}dx\wedge dy \\
&\qquad\quad +\int_{f(|z|=R)}\frac{\partial}{\partial n }\Delta_{w} v^{-}  |dw|\bigg]\\
=&-\int_{(B_1(0))^c}\frac{8\big|\big(|z|\ln|z|\big)_{zz}|f_{z}|^{-1}f_{z}^{-1}\bar{f_{z}}
+\frac{1}{2}|z|\ln|z||f_{z}|^{-1}f_{z}^{-1}\bar{f_{z}}\emph{S}(f)\big|^2}{3|z|\ln|z|}|f_{z}|dx\wedge dy \\
=&-\int_{(B_1(0))^c}\frac{8\bigg|\frac{\ln |z|}{4|z|}\frac{\bar{z}^2}{|z|^2}|f_{z}|^{-1}f_{z}^{-1}\bar{f_{z}}
+\frac{1}{2}|z|\ln|z||f_{z}|^{-1}f_{z}^{-1}\bar{f_{z}}\emph{S}(f)\bigg|^2}{3|z|\ln |z|}|f_{z}|dx\wedge dy\\
=&-\int_{(B_1(0))^c}\bigg[\frac{1}{6}\frac{\ln |z|}{|z|^3|f_z|}
+\frac{2|z|}{3}\ln|z|\frac{|\emph{S}(f)|^2}{|f_z|}-\frac{\ln |z|}{3|z|^3}\frac{z^2\emph{S}(f)+ \overline{z^2\emph{S}(f)} }{|f_z|}\bigg]
dx\wedge dy.
\end{split}
 \end{align}
In polar coordinates,
 \begin{equation*}-\frac{z^2\emph{S}(f)+ \overline{z^2\emph{S}(f)} }{2|f_z|}
 =\bigg(-\frac{\partial^2}{\partial \theta^2}+\frac{r^2}{2}\Delta -r\frac{\partial}{\partial r} \bigg )\frac{1}{|f_z|}.
 \end{equation*}
Hence, for any $ r\in (1,\infty)$,
 \begin{align*}
 &\int_{|z|=r}-\frac{z^2\emph{S}(f)+ \overline{z^2\emph{S}(f)} }{2|f_z|}|dz|\\
 =\, &  \frac{r^2}{2}\int_{|z|=r}\Delta\frac{1}{|f_z|}|dz|
 -r\bigg[\int_{1\leq|z|\leq r}\Delta\frac{1}{|f_z|}dx\wedge dy+\int_{|z|=1}\frac{\partial}{\partial r} \frac{1}{|f_z|}|dz|\bigg].
 \end{align*}
Also for any $ r\in (1,\infty)$, we have
$$\frac{d}{dr}\bigg[\frac1r\int_{|z|=r}\frac{1}{|f_z|}|dz|\bigg]
=\frac1r\bigg[\int_{1\leq|z|\leq r}\Delta\frac{1}{|f_z|}dx\wedge dy+\int_{|z|=1}\frac{\partial}{\partial r} \frac{1}{|f_z|}|dz|\bigg].  $$
Note that
$$F(r)=\int_{1\leq|z|\leq r}\Delta\frac{1}{|f_z|}dx\wedge dy+\int_{|z|=1}\frac{\partial}{\partial r} \frac{1}{|f_z|}|dz|$$
is nondecreasing and
\begin{equation}\label{limit to tilde f}
 \lim_{r\rightarrow\infty}\bigg[\frac{1}{2\pi r}\int_{|z|=r}\frac{1}{|f_z|}|dz|\bigg]
 =|\widetilde{f}_z(0)|, \end{equation}
Then, we must have
$$\int_{1\leq|z|\leq r}\Delta\frac{1}{|f_z|}dx\wedge dy+\int_{|z|=1}\frac{\partial}{\partial r} \frac{1}{|f_z|}|dz|\leq0.$$
Otherwise, there would exist $r_0$ and $\epsilon_{0}>0$ such that $F(r_0)>\epsilon_{0}>0$.
Then, for $r\geq r_0$, we have
$$\frac1r\int_{|z|=r}\frac{1}{|f_z|}|dz|
\geq  \frac{1}{r_0}\int_{|z|=r_0}\frac{1}{|f_z|}|dz|+\epsilon_{0}\ln\frac{r}{r_0}\rightarrow\infty\quad\text{as }r\rightarrow\infty,$$
which contradicts to \eqref{limit to tilde f}.
Therefore, we have, for any $ r \in(1,\infty)$,
 \begin{equation*} \int_{|z|=r}-\frac{z^2\emph{S}(f)+ \overline{z^2\emph{S}(f)} }{2|f_z|}|dz|\geq0.\end{equation*}
Similarly, we have
\begin{equation*}
\frac{d}{dr}\bigg[\frac1r\int_{|z|=r}|f_z||dz|\bigg]
=\frac1r\bigg[\int_{1\leq|z|\leq r}\Delta|f_z|dx\wedge dy+\int_{|z|=1}\frac{\partial}{\partial r} |f_z||dz|\bigg]\leq 0. \end{equation*}
Hence, for any $ r \in(1,\infty)$,
\begin{equation*}\frac{1}{2\pi r}\int_{|z|=r}|f_z||dz|\leq\frac{1}{2\pi}\int_{|z|=1}|f_z||dz|=1.\end{equation*}
Then, we have, for any $ r \in(1,\infty)$,
\begin{equation*}\frac{1}{2\pi}\int_{|z|=1}\frac{1}{|f_z|}|dz|\geq \frac{1}{2\pi r}\int_{|z|=r}\frac{1}{|f_z|}|dz|
\geq2\pi r\bigg[\int_{|z|=r}|f_z||dz|\bigg]^{-1}\geq1.\end{equation*}
Moreover,
\begin{equation}\label{fz=1}\frac{1}{r}\int_{|z|=r}\frac{1}{|f_z|}|dz|=\int_{|z|=1}\frac{1}{|f_z|}|dz|=2\pi
\quad\text{for any } r \in(1,\infty),\end{equation}
if and only if
\begin{equation*}\frac{1}{|f_z|}=|f_z|\quad\text{for any } r \in(1,\infty).\end{equation*}
This implies $f_z\equiv e^{i\varphi}$ for some constant $\varphi$.
Then, we have
\begin{align*}\int_{\sigma_2}c_3^{-}dl
  &\leq -\int_{(B_1(0))^c}\frac{1}{6}\frac{\ln |z|}{|z|^3|f_z|} dx\wedge dy
  = -\int_{1}^{\infty}\bigg(\int_{0}^{2\pi}\frac{1}{6}\frac{\ln r}{r^3|f_z|}d\theta \bigg)rdr \\
 &\leq  -\int_{1}^{\infty}\frac{\pi}{3}\frac{|\ln r|}{r^3}rdr
  =-\frac{\pi}{3}.
\end{align*}
If $\int_{\sigma_2}c_3^{-}dl<-{\pi}/{3}$, then we have
$$\int_{\sigma_2}c_3dl\leq \int_{\sigma_2}c_3^{-}dl<-\frac{\pi}{6}.$$
If $\int_{\sigma_2}c_3^{-}dl=-{\pi}/{3}$, then \eqref{fz=1} holds.
Therefore, $f=e^{i\varphi}z+b$ for some constants $\varphi$ and $b$. Without loss of generality, we assume $b=0$, then, $f$ is a rotation, $\sigma_2=\partial B_1(0)$ and $\Omega_{2*}=(B_1(0))^c$.
Then, for some $R>0$ sufficiently large,
we have $\Omega\subseteq B_R(0)\backslash \overline{B_{1}(0)}.$
Let $v^{R}$
be the solution of \eqref{eq-MEq-v}-\eqref{eq-MBoundary-v}
in $B_R(0)\backslash \overline{B_{1}(0) }$ and $c_3^{R}$ be the coefficient of the corresponding first global term for $v^{R}$.
Then, we have
$$\int_{\sigma_2}c_3dl\leq \int_{\sigma_2}c_3^{R}dl<-\frac{\pi}{3}.$$
Hence, in both cases, we have $$\int_{\sigma_2}c_3 dl<-\frac{\pi}{3}.$$
Therefore, $$\int_{\sigma_2} dl\int_{\sigma_2}c_3 dl<-\frac{2\pi^2}{3}.$$

This completes the proof.
\end{proof}

Now, we are ready to prove Theorem \ref{not Simply connected domains}.

\begin{proof}[Proof of Theorem \ref{not Simply connected domains}]
Without loss of generality, we assume $\sigma_1$ is the boundary of the unbounded component of
$\mathbb R^2\backslash\overline{ \Omega}$.
 For each $i\geq 2$, let $\Omega_i$ be the bounded domain which is enclosed by $\sigma_1$ and $\sigma_i$.
Then, $\Omega\subseteq\Omega_i$.
Let $v_i$
be the solution of \eqref{eq-MEq-v}-\eqref{eq-MBoundary-v}
in $ \Omega_i$ and $c_3^{i}$ be the coefficient of the first global term for $v_i$.
Then, by the maximum principle, we have
$$\int_{\sigma_1}c_3dl\leq \int_{\sigma_1}c_3^{i}dl,$$
and, for $i=2,...,k$,
$$\int_{\sigma_i}c_3dl\leq \int_{\sigma_i}c_3^{i}dl.$$
By Theorem \ref{Simply connected domains}, we have
$$\int_{\sigma_1}dl\int_{\sigma_1}c_3dl\leq \int_{\sigma_1}dl\int_{\sigma_1}c_3^{i}dl<-\frac{2\pi^2 }{3},$$
and, for $i=2,...,k$,
$$\int_{\sigma_i}dl\int_{\sigma_i}c_3dl\leq\int_{\sigma_i}dl \int_{\sigma_i}c_3^{i}dl<-\frac{2\pi^2 }{3}.$$
This finishes the proof.
\end{proof}

Theorem \ref{Simply connected domains rigidity} is a direct consequence of Theorem \ref{not Simply connected domains}.
We now prove Theorem \ref{k-connected domains rigidity}.

\begin{proof}[Proof of Theorem \ref{k-connected domains rigidity}]
Suppose $\partial \Omega=\bigcup_{i=1}^{k}\sigma_{i}$, where $\sigma_i$ is a simple closed $C^{3,\alpha}$ curve, $i=1,...,k$.
Assume $k\geq2$. By Theorem \ref{not Simply connected domains}, we have, for each $i$,
$$\int_{\sigma_i}c_3dl<-\frac{\frac{2\pi^2}{3}}{\int_{\sigma_i}dl}.$$
Therefore,
$$\int_{\partial \Omega}dl\int_{\partial \Omega}c_3dl
<-\bigg(\sum_{i=1}^{k}\int_{\sigma_i}dl\bigg)
\bigg(\sum_{i=1}^{k} \frac{\frac{2\pi^2}{3}}{\int_{\sigma_i}dl}\bigg)\leq -\frac{2k^2\pi^2}{3}.$$
Therefore, if
$$\frac{3}{2\pi^2}\int_{\partial \Omega}dl\int_{\partial \Omega}c_3dl\geq-l^2,$$
we must have $k<l$.
\end{proof}

To end this section, we point out that the upper bound $-{2\pi^2}/{3}$ is optimal for $2$-connected domains.

\begin{example}\label{optimal for 2 connected doamins}
Set $\Omega=B_{{1}/{R}}(0)\backslash \overline{B_{R}(0)},$ for an arbitrary $ R \in (0,1)$.
Then, the solution of \eqref{eq-MEq-v}-\eqref{eq-MBoundary-v} in $\Omega$ is given by
$$v_{R}=|z|\frac{2}{\pi}\ln\frac{1}{R}\sin\frac{\ln\frac{|z|}{R}}{\frac{2}{\pi}\ln \frac{1}{R}}.$$
Let $c_3^{R}$ be the coefficient of the first global term for $v_R$.
 Then,
 $$2\pi R\int_{\partial B_R(0)}c_3^{R}dl= 2\pi \frac{1}{R}\int_{\partial B_{{1}/{R}}(0)}c_3^{R}dl
 =-\frac{2\pi^2}{3}-\frac{2\pi^2}{3(\frac{2}{\pi}\ln\frac{1}{R})^2}\rightarrow-\frac{2\pi^2}{3}.$$
 as $R\rightarrow0$.
\end{example}

In the next example, we will show that the upper bound $-{2\pi^2}/{3}$
of the normalized integral $\int_{\sigma_i}dl\int_{\sigma_i}{c_3}dl$
is also optimal in general multiply connected domains.

\begin{example}\label{optimal for k connected doamins}
Let $\{p_1,\cdots,p_k \}$ be a collection of finitely many points in $\mathbb R^2$,
with $k\ge 1$,
and set $\Omega=B_R(0)\backslash \bigcup_{i=1}^{k}\overline{ B_r(p_i)}$,
where $R$ is sufficiently large and $r$ is sufficiently small.
Let $v$
be the solutions of \eqref{eq-MEq-v}-\eqref{eq-MBoundary-v}
in $\Omega$ and $c_3$ be the coefficient of the first global term for $v$.
Then, for a fixed small number $\epsilon$, by comparing $v$
with the solution of  \eqref{eq-MEq-v}-\eqref{eq-MBoundary-v} in $B_{\epsilon}(p_i) \backslash \overline{ B_r(p_i)}$
and with the solution in $B_R(0)\backslash \overline{B_{{1}/{\epsilon}}(0)}$, respectively, we have
 $$2\pi R\int_{\partial B_R(0)}c_3dl\rightarrow-\frac{2\pi^2}{3},$$
 and,  for $i=1,..,k$,
 $$2\pi r\int_{\partial B_r(p_i)}c_3 dl\rightarrow-\frac{2\pi^2}{3},$$
as $R\rightarrow\infty$ and $r\rightarrow0$.
\end{example}


\begin{thebibliography}{DG}

\bibitem{ACF1982CMP} L. Andersson, P. Chru\'sciel, H. Friedrich,
\emph{On the regularity of solutions to the Yamabe equation and the existence of
smooth hyperboloidal initial data for Einstein field equations},
Comm. Math. Phys., 149(1992), 587-612.


\bibitem{AM1988DUKE} P. Aviles, R. C. McOwen,
\emph{Complete conformal metrics with negative scalar curvature in compact Riemannian manifolds},
Duke Math. J., 56(1988), 395-398.


\bibitem{ChangGurskyYang2002} S-Y. A. Chang, M. Gursky, P, Yang, {\it An equation of Monge-Amp¨¨re type in conformal geometry, and four-manifolds of positive Ricci curvature}, Ann. of Math. (2) 155(2002), no. 3, 709-787.

\bibitem{ChangGurskyYang2003} S-Y. A. Chang, M. Gursky, P, Yang, {\it A conformally invariant sphere theorem in four dimensions}, Publ. Math. Inst. Hautes ¨¦tudes Sci. No. 98(2003), 105-143.



\bibitem{ChangQingYang2004} S-Y. A. Chang, J. Qing, P, Yang, {\it On the topology of conformally compact Einstein 4-manifolds}, Noncompact problems at the intersection of geometry, analysis, and topology, 49-61, Contemp. Math., 350, Amer. Math. Soc., Providence, RI, 2004.


\bibitem{ChangQingYang2007} S-Y. A. Chang, J. Qing, P. Yang, {\it  On a conformal gap and finiteness theorem for a class of four-manifolds},
Geom. Funct. Anal., 17(2007), 404-434



\bibitem{ChengYau1980CPAM} S.-Y. Cheng, S.-T. Yau,
\emph{On the existence of a complete K\"ahler metric on non-compact complex
manifolds and the regularity of Fefferman equation},
Comm. Pure Appl. Math., 33(1980), 507-544.





\bibitem{M.Gursky1}  M. Gursky, J. Streets, M.  Warren,
\emph{Existence of complete conformal metrics of negative Ricci
curvature on manifolds with boundary}, Cal. Var. \& PDE,
41(2011) 21-43.

\bibitem{Graham1999}C. R. Graham, \emph{Volume and Area Renormalizations for Conformally Compact Einstein Metrics}, Proc. of 19th Winter School in Geometry and Physics, Srni, Czech Rep., Jan. 1999, Rend.circ.mat.palermo Suppl, 1999 (63) :31-42


\bibitem{Graham2017}C. R. Graham, \emph{Volume renormalization for singular Yamabe metrics}, Proc. Amer. Math. Soc. 145 (2017), no. 4, 1781-1792.

\bibitem{HanJiang} Q. Han, X. Jiang,
\emph{Boundary expansions for minimal graphs in the hyperbolic space},
preprint, 2015.

\bibitem{HanShen} Q. Han, W. Shen,
\emph{Boundary Behaviors for Liouville's Equation in planar singular domains},
J. Funct. Anal.,274 (2018) 1790-1824

\bibitem{HanShen2} Q. Han, W. Shen, \emph{\it The Loewner-Nirenberg problem in singular domains},
arxiv:1511.01146v1.

\bibitem{HanShen3} Q. Han, W. Shen,
\emph{On the negativity of Ricci curvatures of complete conformal metrics},
arxiv£º1702.03652v2, 2017.

\bibitem{JianWang2013JDG} H. Jian, X.-J. Wang,
{\it Bernstein theorem and regularity for a class of Monge-Amp\`{e}re equations},
J. Diff. Geom., 93(2013), 431-469.


\bibitem{Kichenassamy2004JFA}
S. Kichenassamy, \emph{Boundary blow-up and degenerate equations},
J. Funct. Anal., 215(2004), 271-289.

\bibitem{Kichenassamy2005JFA}
S. Kichenassamy, \emph{Boundary behavior in the Loewner-Nirenberg problem},
J. Funct. Anal., 222(2005), 98-113.


\bibitem{LeeMelrose1982} J. Lee, R. Melrose,
\emph{Boundary behavior of the complex Monge-Amp\`ere equation},
Acta Math., 148(1982), 159-192.

\bibitem{LiQingShi2017} G. Li, J. Qing, Y. Shi, {\it Gap phenomena and curvature estimates for conformally compact Einstein manifolds}, Trans. Amer. Math. Soc. 369 (2017), no. 6, 4385-4413.




\bibitem{Lin1989Invent} F.-H. Lin,
{\it On the Dirichlet problem for minimal graphs in hyperbolic space},
Invent. Math., 96(1989), 593-612.



\bibitem{Loewner&Nirenberg1974} C. Loewner, L. Nirenberg,
\emph{Partial differential equations invariant under conformal or projective transformations},
Contributions to Analysis, 245-272, Academic Press, New York, 1974.


\bibitem{Mazzeo1991} R. Mazzeo,
\emph{Regularity for the singular Yamabe problem}, Indiana Univ. Math. Journal, 40(1991), 1277-1299.


\end{thebibliography}
\end{document}